%% file: main.tex
\documentclass[dvipsnames]{amsart}
\usepackage[utf8]{inputenc}
\usepackage{graphicx}
\usepackage{amssymb}
\usepackage{amsmath,tensor}
\usepackage{float}
\usepackage{xcolor}
\usepackage[
style=alphabetic,
maxbibnames=9,
doi=false,isbn=false,url=false,
]{biblatex}
\usepackage{mathscinet}

\usepackage{hyperref}
\usepackage{cleveref}
\usepackage{comment}
\usepackage[all]{xy}

\newtheorem{thm}{Theorem}[section] 
\newtheorem{prop}[thm]{Proposition}
\newtheorem{lem}[thm]{Lemma}
\newtheorem{cor}[thm]{Corollary}
\newtheorem{conj}[thm]{Conjecture}

\theoremstyle{definition}
\newtheorem{defn}[thm]{Definition}
\newtheorem{notation}[thm]{Notation}
\newtheorem{ex}[thm]{Example}

\theoremstyle{remark}
\newtheorem{rem}[thm]{Remark}

\newcommand{\mg}{\mathfrak{m}}
\def\R{\mathbb{R}}
\def\Z{\mathbb{Z}}

\newcommand{\al}{\alpha}

\def\hat{\widehat}

\newcommand{\co}{\colon}
\newcommand{\pa}{\partial}

\DeclareMathOperator{\Hom}{Hom}
\DeclareMathOperator{\End}{End}

\DeclareMathOperator{\id}{id}

\newcommand{\im}{\operatorname{Im}}
\newcommand{\inv}{^{-1}}

\newcommand{\mi}{\underline}
\newcommand{\ma}{\overline}
\newcommand{\expr}{\leftrightharpoons}

\newcommand{\Dem}{\mathcal{D}}

\newcommand{\SC}{\mathcal{SC}}
\newcommand{\mt}{\emptyset}
\newcommand{\Sym}{\operatorname{Sym}}

\DeclareMathOperator{\leftdes}{LD}
\DeclareMathOperator{\rightdes}{RD}
\DeclareMathOperator{\leftred}{LR}
\DeclareMathOperator{\rightred}{RR}
\DeclareMathOperator{\core}{core}

\DeclareMathOperator{\rank}{rank}
\title{Demazure operators for double cosets}
\author{Ben Elias, Hankyung Ko, Nicolas Libedinsky, Leonardo Patimo}

\begin{document}
\maketitle

\begin{abstract} 
For any Coxeter system, and any double coset for two standard parabolic subgroups, we introduce a Demazure operator. These operators form a basis for morphism spaces in a category we call the nilCoxeter category, and we also present this category by generators and relations.
We prove a generalization to this context of Demazure's celebrated theorem on Frobenius extensions. This generalized theorem serves as a criterion for ensuring the proper behavior of singular Soergel bimodules.
\end{abstract}


\input{Introduction}

\input{ReducedExpressions}

\input{Demazure}

\input{Frobenius}

\printbibliography
\end{document}

%% file: Introduction.tex
\section{Introduction}

This paper tells a completely algebraic story, though our motivations come from geometry and categorification. For sake of simplicity we focus on type $A$ in this introduction, but we work with arbitrary Coxeter groups in the body of the text, with connections to geometry (as stated) only for Weyl groups.

Let $R = \mathbb{C}[x_1, \ldots, x_n]$ be the polynomial ring in $n$ variables, equipped with its action of $W = S_n$. Let $R^W$ denote the subring of $W$-invariant polynomials, and for each $1 \leq i \leq n-1$, let $R^i$ denote the subring of polynomials invariant under the simple reflection $s_i := (i,i+1)$. There is a \emph{Demazure operator} or \emph{divided difference operator}
\begin{equation} \label{eq:paidefintro} \pa_i \co R \to R^i, \qquad \pa_i(f) = \frac{f - s_i(f)}{x_i - x_{i+1}}. \end{equation}
In similar fashion, one also has an operator
\begin{equation} \pa_W \co R \to R^W, \qquad \pa_W(f) = \frac{\sum_{w \in W} (-1)^{\ell(w)} w(f)}{\Pi_{i < j} (x_i - x_j)}. \end{equation}
These operators are now ubiquitous in combinatorics and representation theory.

Demazure operators appear when considering pushforwards in cohomology, along forgetful maps between partial flag varieties. There is also a version using equivariant cohomology of a point, for various parabolic subgroups. The geometric underpinning of Demazure operators is well-trodden; instead of rehashing it here, we refer the reader unfamiliar with this story to \cite{BGG,Brion}.

The key observation with which we begin our story is a foundational result of Demazure \cite{Demazure}: that $\pa_W$ lies within the subalgebra generated by $\pa_i$. To make this
statement more precise, let $\iota_i \co R^i \to R$ denote the inclusion map, and let $D_i := \iota_i \circ \pa_i$. Then $D_i$ satisfies the same formula as \eqref{eq:paidefintro},
but with different codomain, and the operators $D_i$ can be composed. Similarly, we set $D_W \co R \to R$ to be $\pa_W$ composed with the inclusion map $\iota_W \co R^W \to R$. The
subalgebra in $\End_{R^W}(R)$ generated by $D_i$ is called the \emph{nilCoxeter algebra}. Demazure proved a presentation for the nilCoxeter algebra: the relations are
\begin{equation} D_i^2 = 0, \qquad D_i D_{i+1} D_i = D_{i+1} D_i D_{i+1}, \qquad D_i D_j = D_j D_i \text{ if } j \ne i \pm 1. \end{equation} 
Because these relations are similar to
those of the symmetric group itself (e.g. this presentation falls into the framework of generalized Hecke algebra, see \cite[Chapter 7]{Humphreys}), the nilCoxeter algebra has a
basis $\{D_w\}_{w \in W}$ where $D_w$ is defined by applying various operators $D_i$ along a reduced expression for $w$. Moreover, Demazure proves that $D_W = D_{w_0}$, the basis
element associated to the longest element of $W$.

For any subset $I \subset S = \{1, \ldots, n-1\}$ of the simple reflections, there is a parabolic subgroup $W_I$, and a subring $R^I$ of $W_I$-invariant polynomials. By
the same token, one has an operator $\pa_I \co R \to R^I$ and $D_I \co R \to R$, and we have $D_I= D_{w_I}$ where $w_I$ is the longest element of $W_I$.

By rephrasing $\pa_i$ as an endomorphism $D_i$ of $R$, one has rudely cut the ring $R^i$ out of the picture. But $D_i = \iota_i \circ \pa_i$ is a composition of two distinct maps of $R^W$-modules. Geometrically, $\pa_i$ is a proper pushforward in cohomology, while $\iota_i$ is a pullback. When a linear transformation factors, it is often helpful to zoom in on the compositional factors independently. Indeed, Demazure's relation $D_i^2 = 0$ is a consequence of a more zoomed-in relation: $\pa_i \circ \iota_i = 0$.
The braid relations for $D_i$ are also consequences of more zoomed-in relations, called the \emph{switchback relations} in \cite{EKo}.

In this paper we are interested in studying the following category, a generalization of the nilCoxeter algebra. It describes all linear maps one could obtain by composing pullbacks and pushforwards between partial flag varieties.

\begin{defn} Let $\Dem$ be the following (non-full) subcategory of finite dimensional vector spaces over $\mathbb{C}$. Its objects are $R^I$ for $I \subset S$. Its morphisms are generated by inclusion maps $\iota^I_J \co R^J \to R^I$ and Demazure operators $\pa^I_J \co R^I \to R^J$, for $I \subset J$. \end{defn}
	
Above, $\pa^I_J$ is the restriction of $D_{w_J w_I^{-1}}$ to $R^I$; one can show its image lies in $R^J$.

Before motivating this category, let us state the main results of \Cref{sec:demazure}, which are analogous to those of Demazure but significantly more refined. We provide:
\begin{itemize}
\item A presentation of $\Dem$ by generators (namely $\iota^I_J$ and $\pa^I_J$) and relations,
\item An operator $\pa_p : R^J \to R^I$ for each double coset $p \in W_I \backslash W / W_J$,
\item A proof that $\{\pa_p\}$ is a basis for $\Hom_{\Dem}(R^J,R^I)$, as $p$ ranges among such double cosets.
\end{itemize}

\begin{rem} The map $\pa^I_J$ is $\pa_q$, where $q \in W_J \backslash W / W_I$ is the double coset containing the identity. The map $\iota^I_J$ is $\pa_p$, where $p \in W_I \backslash W / W_J$ is the double coset containing the identity (see Example \ref{ex:3.10}). \end{rem}

\begin{rem} A simple variant on Demazure's presentation, replacing $D_i^2 = 0$ with $D_i^2 = D_i$, will describe pushforwards and pullbacks in equivariant $K$-theory instead of equivariant homology (of a point). Everything stated above holds in the $K$-theoretic context as well, mutatis mutandis. However, most statements below are special to the homology setting. \end{rem}

\begin{rem} We use different and more streamlined notation in the body of the paper. The category $\Dem$ above matches a category $\Dem'$ in the text, while $\Dem$ is defined differently (though they are proven to be equivalent). \end{rem}

A composition of the generating morphisms can be parametrized by a word $[I_0, I_1, \ldots, I_d]$ in the subsets of $S$, where either $I_i \subset I_{i+1}$ or $I_i \supset I_{i+1}$ for each $0 \le i < d$. Such a word is called a \emph{(singular) expression} or a \emph{double coset expression}.

An ordinary expression $(s_{i_1}, s_{i_2}, \ldots, s_{i_d})$ is reduced if and only if the composition $D_{i_1} D_{i_2} \cdots D_{i_d}$ is nonzero in the nilCoxeter algebra, in
which case it is equal to $D_w$ for some $w$. One might ask a similar question in $\Dem$: which compositions of the generating maps $\pa^I_J$ and $\iota^I_J$ are nonzero? This
gives a valid notion of a \emph{reduced expression} for a double coset.

Reduced expressions for double cosets were introduced in Williamson's thesis \cite{WillThesis}, where one purpose was to produce resolutions of singularities in orbit
closures for partial flag varieties (analogous to Bott-Samelson resolutions, associated to ordinary reduced expressions, in the full flag variety). The theory of reduced expressions was greatly expanded in recent work \cite{EKo}, and several equivalent definitions are given which are easier to work with. In this paper we show that the definition of reduced expressions given loosely above agrees with the definitions in \cite{WillThesis,EKo}. 
In \Cref{sec:singularrex} we recall from \cite{EKo} the results we need.

\begin{rem} Given a double coset expression, one can also associate a singular Bott-Samelson bimodule, as in Williamson's thesis \cite{WillThesis}. Each reduced expression for a double coset $p$ gives a bimodule for which the indecomposable singular Soergel bimodule $B_p$ is a direct summand. In subsequent work, we will use the Demazure operators of this paper to explicitly construct the inclusion from $B_p$ to the Bott-Samelson bimodule of certain reduced expressions, when $p$ is a maximal double coset in a finite parabolic subgroup.
\end{rem}

One of the primary uses of Demazure operators is the following. There is an $R^W$-bilinear pairing $R \times R \to R^W$, where
\begin{equation} (f,g) := \pa_W(fg). \end{equation}
Demazure \cite{Demazure} proves that this is a perfect pairing by constructing dual bases. In modern parlance, this makes $R$ into a \emph{Frobenius extension} over $R^W$; one implication is that induction and restriction between $R$-modules and $R^W$-modules are 
biadjoint functors (up to grading shift). Similar statements can be made about the extension $R^I \subset R$, or even the extension $R^J \subset R^I$ when $I \subset J$. All of these Frobenius structures are shadows of Poincar\'{e} duality for proper maps between the corresponding partial flag varieties.

Suppose $i \ne j \in S$. Here is a statement which is not about the composition of operators $D_i D_j$, but really about the composition $\pa_i \circ\iota_j$: we can choose dual bases
for $R$ over $R^i$, such that one of the bases lives in $R^j$. More generally, the following property was called $(\star)$ in \cite{ESWFrob}: for any $I, J \subset S$, one can
choose dual bases for $R^{I \cap J}$ over $R^I$ such that one of the bases lives in $R^J$. The condition $(\star)$ was essential to the proper functioning of a diagrammatic calculus developed in \cite{ESWFrob}, for computing natural transformations between compositions of induction and restriction functors. These compositions of functors are isomorphic to taking tensor product with certain bimodules, called \emph{singular Bott-Samelson bimodules}.

In subsequent work \cite{KELP3} we will further develop this diagrammatic calculus, providing a basis for morphisms between singular Bott-Samelson bimodules, and a crucial input to
our result is the main result of \Cref{sec:frob}, \Cref{thm:dualbasisinimage}. This can be viewed as a generalization of $(\star)$ for each double coset in $W_I \backslash W / W_J$. \Cref{thm:dualbasisinimage},
when applied to the double coset $p \in W_I \backslash W /W_J$ containing the identity, is exactly the condition $(\star)$, but for other cosets it is a similar but more technical
statement. To state it precisely involves knowing additional attributes of double cosets (e.g. redundancy subgroups), so we do not state the result in the introduction.

{\bf Acknowledgments.} NL was partially supported by FONDECYT-ANID grant 1230247. BE was partially supported by NSF grant DMS-2201387. 

%% file: ReducedExpressions.tex
\section{Singular reduced expressions and the Matsumoto theorem} \label{sec:singularrex}

This chapter is a review of the main results of \cite{EKo} on singular expressions.
The only thing in this chapter which is not contained in \cite{EKo} is Notation \ref{dotforcosets}.

\subsection{Some notation} \label{ssec:notation}
 
Throughout we fix a Coxeter system $(W,S)$. If $I\subset S$, then $W_I$ is the parabolic subgroup generated by $I$. If $W_I$ is finite we call $I$ \emph{finitary}, and we denote its longest element by $w_I$. We let $\ell$ denote the length function on $W$ and we set $\ell(I):=\ell(w_I)$.

If no confusion is possible, we use shorthand for subsets $I \subset S$. We might write $stu$ rather than $\{s,t,u\} \subset S$, and $\hat{s}$ instead of $S \setminus \{s\}$. When $S = \{s_1, \ldots, s_n\}$, we may write $\{2,4,5\}$ or $245$ rather than $\{s_2, s_4, s_5\}$, and $\hat{2}$ instead of $S \setminus \{s_2\}$. If $s \notin I$ then we write $Is$ for $I \sqcup \{s\}$ and if $s \in I$ we write $I \setminus s$ for $I \setminus \{s\}$.

For $J, I\subset S$ finitary, a \emph{$(I,J)$-coset} is an element $p$ in 
 $W_{I}\backslash W/W_{J}$. When we write ``the coset $p$" we  mean the triple $(p,I,J)$. It might happen that  $(p,I,J)\neq (p',I',J')$, even though $p=p'$ as subsets of $W$, and we distinguish between $p$ and $p'$ in this case. If $p$ is a $(I,J)$-coset we denote by $\overline{p}\in W$ and $\underline{p}\in W$ the unique maximal and minimal elements in the Bruhat order in the set $p$.

For $x, y, z \in W$ we write $x.y = z$ to indicate the statement that $xy = z$ and $\ell(x) + \ell(y) = \ell(z)$. Thus we write $x.y = xy$ to imply that the lengths add when $x$ and $y$ are multiplied.  We may say that $x.y$ is a \emph{reduced composition}, in analogy to reduced expressions.

\subsection{Expressions and reduced expressions}
 
 \begin{defn}  A (singular) \emph{multistep  expression} is a sequence of finitary subsets of $S$ of the form 
 \begin{equation}\label{multistep}
 L_{\bullet}=[[ I_0\subset K_1\supset I_1\subset K_2 \supset \cdots \subset K_m\supset I_m]].
 \end{equation}
 By convention, if we write a multistep expression as $[[K_1 \supset I_1 \subset \cdots]]$, this means that $I_0 = K_1$,
and similarly $[[\cdots \subset K_m]]$ means that $I_m = K_m$. \end{defn}

Internal equalities in multistep expressions can be absorbed harmlessly. For example, the multistep expressions
\begin{equation} [[ \ldots K_i \supset I_i = K_{i+1} \supset I_{i+1} \subset \ldots ]] \quad \text{ and } \quad [[ \ldots K_i \supset I_{i+1} \subset \ldots ]] \end{equation}
are equivalent for all practical purposes, such as for the formulas \eqref{reduced} and \eqref{represents}. One may as well assume that multistep expressions have no internal equalities.

\begin{defn} A (singular) \emph{(singlestep) expression} is a sequence of finitary subsets of $S$ of the form
\begin{equation} I_{\bullet} = [I_0, I_1, \ldots, I_d], \end{equation} such that, for each $1 \le k \le d$, either $I_k = I_{k-1}s$ or $I_k = I_{k-1} \setminus s$ for some $s \in S$. The number $d$ is called the \emph{width} of $I_{\bullet}$. \end{defn}

To each singlestep expression one can associate one\footnote{Or many if one allows internal equalities in \eqref{multistep}.} multistep  expression by remembering its local maxima and minima. Singlestep expressions are more useful for purposes of generators and relations, while multistep expressions get to the heart of the matter faster.

\begin{defn} Let $L_{\bullet}$ be a multistep expression of the form \eqref{multistep}.
Temporarily, let $w_{L_{\bullet}} \in W$ denote the element
\begin{equation} w_{L_{\bullet}} := w_{K_1}w_{I_1}^{-1}w_{K_2}\cdots w_{I_{m-1}}^{-1}w_{K_m}. \end{equation}
We say that $L_\bullet$ is \emph{reduced} if
 \begin{equation}\label{reduced}
\ell(w_{L_{\bullet}})=\ell(K_1)-\ell(I_1)+\ell(K_2)-\cdots -\ell(I_{m-1})+\ell(K_m)..
  \end{equation}
In this case we say that $L_{\bullet}$ \emph{expresses} the $(I_0, I_m)$-coset
\begin{equation}\label{represents}
p = W_{I_0}w_{L_{\bullet}}W_{I_m},    
\end{equation}
and we write $L_{\bullet} \expr p$. We write $L_{\bullet} \expr M_{\bullet}$ if  $L_{\bullet}$ and $M_{\bullet}$ express the same double coset $p$.
We say that a (singlestep) expression is \emph{reduced} if the corresponding multistep expression is. \end{defn}

Let us note  that if $L_{\bullet}$ is a reduced expression for $p$, then the maximal element $\ma{p}$ satisfies
\begin{equation} \label{map} \ma{p} = w_{L_{\bullet}} = w_{K_1} . (w_{I_1}^{-1}w_{K_2}) . (\cdots) . (w_{I_{m-1}}^{-1}w_{K_m}). \end{equation}
We stop using the notation $w_{L_{\bullet}}$, and instead start using $\ma{p}$.

By \cite[Proposition 2.31]{EKo} (building\footnote{Part of the work done in \cite{EKo} is to prove that the definition of a reduced expression given here, first stated in \cite{EKo}, is equivalent to the more technical definition given by Williamson.} on a result \cite[Proposition 1.3.4]{WillThesis} of Williamson), any double coset $p$ has a reduced expression.

\begin{rem} In formulas like the above we often write $w_K w_I^{-1}$ when $I \subset K$. Since $w_I$ is an involution, the inverses are not necessary. We include them when they are helpful for computing the length of an element, e.g. $(w_K w_I^{-1}) . w_I = w_K$. \end{rem}

\begin{rem} For more discussion of non-reduced expressions, see \S\ref{ssec:nonreduced}. \end{rem}

\subsection{Concatenation}

\begin{defn} If $I_{\bullet} = [I_0, \ldots, I_d]$ and $K_{\bullet} = [K_0, \ldots, K_e]$ then these expressions are \emph{composable} if $I_d = K_0$, in which case their composition or concatenation is
\[ I_{\bullet} \circ K_{\bullet} := [I_0, \ldots, I_d, K_1, \ldots, K_e]= [I_0, \ldots, I_{d-1}, K_0,  \ldots, K_e].\] \end{defn}

\begin{prop}[{\cite[Proposition 4.3]{EKo}}]\label{prop:concat} Let $I_{\bullet} \expr p$ and $K_{\bullet} \expr q$ be composable reduced expressions, with $J := K_0 = I_d$. Then $I_{\bullet} \circ K_{\bullet}$ is reduced if and only if
\[ \ma{p} w_J^{-1} \ma{q} = \ma{p} . (w_J^{-1} \ma{q}) = (\ma{p} w_J^{-1}) . \ma{q}. \] 
Moreover\footnote{This last statement is only stated implicitly in \cite{EKo}, and follows immediately from unraveling the equality $p * q = r$.}, $I_{\bullet} \circ K_{\bullet} \expr r$ where 
\begin{equation}\label{upper} \ma{r} = \ma{p} w_J^{-1} \ma{q}. \end{equation}
\end{prop}

\begin{notation} \label{dotforcosets} This notation is not found in \cite{EKo}. For an $(I,J)$-coset $p$, a $(J,K)$-coset $q$, and a $(I,K)$-coset $r$, let us write $p . q = r$ if we can find $I_{\bullet}$ and $K_{\bullet}$ such that
\[ I_{\bullet} \expr p, \quad K_{\bullet} \expr q, \quad I_{\bullet} \circ K_{\bullet} \expr r\]
are all reduced expressions. We say that $r$ is a \emph{reduced composition} of $p$ and $q$. In this case we also write $I_{\bullet} \circ K_{\bullet} = I_{\bullet} . K_{\bullet}$. By the previous proposition, $p . q = r$ if and only if $\ma{r} = \ma{p}. (w_J^{-1} \ma{q}) = (\ma{p} w_J^{-1}) . \ma{q}$.

We also permit ourselves to mix and match cosets and their reduced expressions using this notation. If $I_{\bullet}$ is a reduced expression for $p$, then $I_{\bullet} . q$  represents a reduced composition of $I_{\bullet}$ with some reduced expression for $q$ (the choice being, presumably, not relevant). We may write $I_{\bullet} . q \expr p . K_{\bullet} \expr r$ as well.
\end{notation}

\begin{defn}
We say that $I_{\bullet}$ is a \emph{contiguous subword} of $L_{\bullet}$ if there exist expressions $J_{\bullet}$ and $K_{\bullet}$ such that $L_{\bullet} = J_{\bullet} \circ I_{\bullet} \circ K_{\bullet}$.
\end{defn}

\begin{prop}[{\cite[Proposition 3.12]{EKo}}] \label{prop:reverse}  A contiguous subword of a reduced expression is reduced. Reversing the order of a reduced expression yields a reduced expression.\end{prop}

 \subsection{Redundancy and the core} \label{ssec:core}

\begin{defn} Let $p$ be a (finitary) $(I,J)$-coset. The \emph{left redundancy of $p$}, denoted $\leftred(p)$, is the subset of $I$ defined as follows:
\begin{equation} \leftred(p) := I \cap \mi{p} J \mi{p}^{-1}. \end{equation}
The \emph{right redundancy of $p$} is a subset of $J$ defined similarly:
\begin{equation} \rightred(p) := J \cap \mi{p}^{-1} I \mi{p}. \end{equation}
\end{defn}

There is a surjective map $W_I \times W_J \to p$ sending $(x,y) \mapsto x \mi{p} y$. The fibers of this map are torsors over the (left or right) redundancy. We have
\begin{equation} \label{mapmip} \ma{p} = (w_I w_{\leftred(p)}^{-1}) . \mi{p} . w_J = w_I . \mi{p} . (w_{\rightred(p)}^{-1} w_J).  \end{equation}
For details, see \cite[Lemma 2.12]{EKo}. The ideas above are originally due to Kilmoyer and Howlett.

\begin{defn}\label{core} Let $p$ be a $(I,J)$-coset. The \emph{core} of $p$ is the $(\leftred(p),\rightred(p))$-coset $p^{\core}$ with minimal element $\mi{p^{\core}}=\mi{p}$. \end{defn}

In \cite[Lemma 4.27]{EKo}, it is proven that the core is well-defined, and that $\leftred(p^{\core}) = \leftred(p)$ and $\rightred(p^{\core}) = \rightred(p)$. The main theorem of Section 4.9 in \cite{EKo} is that $p$ has a reduced expression which factors through its core.

\begin{thm}[{\cite[Cor. 4.27 and Prop. 4.28]{EKo}}]\label{thmA}\label{thm:core} Let $p$ be a $(I,J)$-coset. Then $p$ has a reduced expression of the form $[I,I\setminus s,\ldots,J]$ if and only if $s\not \in \leftred(p)$. Moreover, for any reduced expression $M_{\bullet}$ of $p^{\core}$, the (multistep) expression
\begin{equation} \label{throughcore} [[I \supset \leftred(p)]] \circ M_{\bullet} \circ [[\rightred(p) \subset J]] \end{equation}
is a reduced expression for $p$. \end{thm}

Using Notation \ref{dotforcosets} we may write
\begin{equation} p \expr [[I \supset \leftred(p)]] . p^{\core} . [[\rightred(p) \subset J]]. \end{equation}

\subsection{Addable and removable elements}

\begin{defn} For an element $w \in W$, let $\leftdes(w) = \{ s \in S \mid sw < w \}$, commonly called the \emph{left descent set} of $w$. Similarly, $\rightdes(w) = \{ s \in S \mid ws < w \}$ is the \emph{right descent set}.

For a double coset $p$, let $\leftdes(p)$ and $\rightdes(p)$, the \emph{left and right descent sets} of $p$, denote the (ordinary) left and right descent sets of $\ma{p}$. \end{defn}

Let $p$ be an $(I,J)$-coset. By \cite[Lemma 2.12(5)]{EKo}, $I \subset \leftdes(p)$ and $J \subset \rightdes(p)$. Moreover, for any $w \in p$, $w = \ma{p}$ if and only if $I \subset \leftdes(w)$ and $J \subset \rightdes(w)$.

\begin{thm}[{\cite[Cor. 4.20 and Prop. 4.21]{EKo}}] \label{thm:addremove} Let $p$ be an $(I,J)$-coset. Then $p$ has a reduced expression of the form $[I,Is, \ldots, J]$ if and only if $s \in \leftdes(p) \setminus I$. It has a reduced expression of the form $[I,\ldots,Js,J]$ if and only if $s \in \rightdes(p) \setminus J$. Moreover, $p$ has some reduced expression of the form 
\[[[I \subset \leftdes(p)]] . N_\bullet.[[\rightdes(p)\supset J]],\] where $N_\bullet$ is any reduced expression of the $(\leftdes(p),\rightdes(p))$-coset containing $p$.
\end{thm}

\begin{proof} See \cite[Corollaries 4.20 and 4.27, Propositions 4.21 and 4.28]{EKo}. The statement about $[I,\ldots,Js,J]$ is not found in \cite{EKo}, but can be easily deduced from the statement about $[I,Is,\ldots,J]$ by applying Proposition \ref{prop:reverse} to reverse the order of the word. \end{proof}

In other words, every double coset $p$ admits a reduced expression which begins and ends going up as much as possible, i.e. to the descent sets of $p$. On the other hand, by \Cref{thm:core} every double coset $p$ admits a reduced expression which begins and ends going down as much as possible, i.e. to the redundancy sets of $p$.

\subsection{Braid relations and Matsumoto's theorem}

Another convenient way to keep track of an expression $I_{\bullet} = [I_0, \ldots, I_d]$ is to write down not the parabolic subsets $I_k$, but the sequence of simple reflections which were added and removed. If $I_k = I_{k-1} s$ then we write $+s$, and if $I_k = I_{k-1} \setminus s$ then we write $-s$. For example, we have
\begin{equation} \label{plusminusnotation} [I + s + t - u + v] := [I, Is, Ist, Ist\setminus u, Istv \setminus u]. \end{equation}
By convention, this notation implies that $s, t \notin I$ and $u \in Ist$, and that all the subsets of $S$ appearing in the expression are finitary.

Now we list the \emph{singular braid relations}, which are local transformations one can apply to contiguous subwords of reduced expressions.
\begin{subequations} \label{braidrelns}
\begin{itemize}
    \item The \emph{up-up relation} is 
\begin{equation} \label{upup} [L+s+t] \expr [L+t+s] \end{equation}
for any $s, t \notin L$ with $Lst$ finitary.
\item The \emph{down-down relation} is
\begin{equation} \label{downdown} [L-s-t] \expr [L-t-s]\end{equation}
for any $s, t \in L$ with $L$ finitary.
\item The \emph{commuting (switchback) relation} is
\begin{equation} [L + s - t] \expr [L - t + s] \end{equation}
for any $s \notin L$ and $t \in L$ satisfying
\begin{itemize} \item $Ls$ is finitary, and
\item $s$ and $t$ are in different connected components of $Ls$. \end{itemize}
\item The \emph{switchback relation} is
\begin{equation} \label{switchback} [L + s - t] \expr [L - u_1 + u_0 - u_2 + u_1 - \cdots + u_{\delta-2} - u_{\delta} + u_{\delta-1}], \end{equation}
 for $s \notin L$, $Ls$ finitary, and $t \in L$ satisfying that $t \ne w_{Ls}  s  w_{Ls}$. Here $(u_i)$ is the ``rotation sequence'' (see below) associated to the triple $(Ls,s,t)$, with $u_0 = s$ and $u_{\delta} = t$.
\end{itemize}
\end{subequations}

The switchback relation is complicated enough to merit additional discussion, but let us first state the main theorem about braid relations.

\begin{thm}[{Singular Matsumoto Theorem \cite[Thm. 5.30]{EKo}}] \label{thm:matsumoto}  Any two reduced expressions for the same double coset are related by a sequence of singular braid relations. \end{thm}

Let us describe the switchback relation succinctly in type $A$. If $s$ and $t$ are not in the same connected component of $Ls$ then the switchback relation becomes merely the commuting relation. First consider the case when $Ls = S$ and $W = S_n$.

\begin{ex} Let $W = S_n$ and $S = \{s_1, \ldots, s_{n-1}\}$ be the usual simple reflections. Let $s = s_a$ and $t = s_b$, and $L = S \setminus s_a$. If $a + b = n$ then $t = w_S s  w_S$ and there is no switchback relation. Otherwise, if $a+b < n$ we set $c = a+b$, and if $a + b > n$ we set $c = a+b - n$. Then the switchback relation is
\begin{equation} \label{typeAswitchback} [L + s_a - s_b] \expr [L - s_c + s_a - s_b + s_c]. \end{equation}
Both sides are reduced expressions for the $(\hat{s_a}, \hat{s_b})$-coset containing $w_S$.
\end{ex}

The general case in type $A$ is below.

\begin{ex} Let $W = S_m$ and $Ls$ be an arbitrary parabolic subset. Let $I$ denote the connected component of $Ls$ containing $s$, and suppose $t \in I$. Then $W_I \cong S_n$ for some $n \le m$, and the switchback relation for $s$ and $t$ is precisely as in the previous example. \end{ex}

In arbitrary types, the rotation sequence is a particular walk around the Dynkin diagram, using conjugation by longest elements of parabolic subgroups, see \cite[Def. 5.10]{EKo}. The number $\delta$ is $1$ for the case of the commuting switchback relation. Otherwise, $\delta$ is equal to $2$ in type $A$, at most $3$ in types $BCD$, and caps out at $\delta = 11$ in type $H_4$. For explicit details on the switchback relation see \cite[Chapter 5]{EKo}.

\subsection{Non-reduced expressions} \label{ssec:nonreduced}

In \cite[Section 1.3]{EKo} the authors define the singular Coxeter monoid, a category where the morphisms are double cosets, and where a singlestep expression represents a composition of generating morphisms. In this way, any expression expresses some double coset, see \cite[Def. 1.27 and Eq. 1.15]{EKo}. When the expression is not reduced, the formula for this double coset might not agree with \eqref{represents}. Reduced expressions are exactly the expressions of minimal length expressing a coset, for a notion of length defined in \cite[Section 3.4]{EKo}.



The \emph{(singular) $*$-quadratic relation} is
\begin{equation} \label{starquad} [L - s + s] \expr [L] \end{equation}
for $s \in L$. The left-hand side is a non-reduced expression, expressing the same double coset as the right-hand side. The main result on non-reduced expressions is below.

\begin{thm}[{\cite[Thm. 5.31]{EKo}}] \label{thm:notreduced}  Any expression $I_{\bullet}$ can be transformed into a reduced expression by applying braid relations and applying the $*$-quadratic relation in only one direction, replacing the left-hand side of \eqref{starquad} with the right-hand side. \end{thm}

%% file: Demazure.tex
\section{Demazure operators and double cosets} \label{sec:demazure}

The goal of this chapter is to define Demazure operators associated to double cosets, and relate them to the theory of singular reduced expressions.

\subsection{Realizations}

Let $(W,S)$ be a Coxeter system and $\Bbbk$ be a field\footnote{More generally one can let $\Bbbk$ be a domain, and replace `finite-dimensional' with `free of finite rank' throughout.}. Let $V$ be a \emph{realization} of $W$ over $\Bbbk$. In particular $V$ is a finite dimensional $\Bbbk$-vector space, equipped with a set of roots $\{\al_s\}_{s \in S}$ in $V$ and coroots $\{\al_s^{\vee}\}_{s \in S}$ in $V^* = \Hom_{\Bbbk}(V,\Bbbk)$, for which the formula
\begin{equation} s(v) = v - \langle \al_s^\vee, v\rangle \al_s \end{equation}
defines an action of $W$ on $V$. A realization is one way to generalize the reflection representation of $V$. See \cite[Section 3.1]{Soergelcalculus} for additional technical details.

\begin{ex} Because it is easier than the reflection representation, we encourage the reader to think about the permutation representation of $S_n$ as their running example. Thus $W = S_n$ acts in the usual way on $V := \Bbbk\{x_1,\ldots,x_n\}$, and $\al_{s_i} = x_i - x_{i+1}$. \end{ex}

We make three technical assumptions throughout the body of the text.

The first assumption is that our realization is faithful, i.e. $V$ is a faithful representation of $W$. Faithfulness is a requirement for the algebraic category of (singular) Soergel bimodules to behave as expected (i.e. for the Soergel--Williamson Hom formula to hold). When the realization is not faithful, the diagrammatic version of the Hecke category should still behave well, though the diagrammatic version of the singular Hecke category has not yet been fully defined.

The second assumption is that our realization is \emph{balanced}, see \cite[Def. 3.6]{Bendihedral}. Let $a_{st} := \langle \al_s^\vee, \al_t \rangle$. For example, when $m_{st} = 3$, being balanced is equivalent to
\begin{equation} a_{st} = a_{ts} = -1. \end{equation}
In a general realization with $m_{st} = 3$, one might have $a_{st} = q$ and $a_{ts} = q^{-1}$ for any invertible element $q \in \Bbbk$. There are important unbalanced realizations, see \cite{EQuantumI}. Our techniques should apply to the unbalanced case as well, but would require a significant amount of extra bookkeeping, see e.g. \cite[Chapter 5]{EQuantumI}.

The third assumption, \emph{generalized Demazure surjectivity}, is discussed throughout this chapter (see Definition~\ref{def.Demsurj}.) It is essential for the proper behavior and even the well-definedness of both the algebraic and diagrammatic Hecke categories.

What happens when these assumptions fail will be addressed in \S\ref{subsec:nastier}.

\subsection{Basics of Demazure operators}

Let $R = \Sym(V)$ be the symmetric algebra of $V$, graded so that $\deg V = 2$. For each $I \subset S$ finitary, one can consider the subring $R^I$ of $W_I$-invariant polynomials. We permit ourselves the usual shorthand from \S\ref{ssec:notation}, using notation like $R^{st}$ instead of $R^{\{s,t\}}$ and $R^{Is}$ instead of $R^{I \sqcup \{s\}}$.

For each simple reflection one can define an $R^s$-linear map $\pa_s \co R \to R^s$ of degree $-2$ via the formula
\begin{equation} \pa_s(f) := \frac{f - sf}{\al_s}. \end{equation}
These well-known operators are called \emph{Demazure operators} or \emph{divided difference operators}. Demazure operators  satisfy the \emph{nil-quadratic relation}
\begin{equation} \label{nilquad} \pa_s^2 = 0.\end{equation}
They also (by the balanced assumption) satisfy the braid relations: for example, when $m_{st} = 3$ one has
\begin{equation} \label{dembraid} \pa_s \pa_t \pa_s = \pa_t \pa_s \pa_t. \end{equation}

The Demazure operators can all be viewed as living inside $\End_{\Bbbk}(R)$, where they generate the so-called \emph{nilCoxeter algebra}. This graded algebra has a presentation with generators $\{\pa_s\}_{s \in S}$ (all of degree $-2$). The relations are the braid relations and the nil-quadratic relations. It also has a basis $\{\pa_x\}$ indexed by $x \in W$, where
\[ \pa_x := \pa_{s_1} \pa_{s_2} \cdots \pa_{s_d} \]
using any reduced expression $x = s_1 s_2 \cdots s_d$. These operators satisfy the formula
\begin{equation} \label{composeDem} \pa_x \circ \pa_y = \begin{cases} \pa_{xy} & \text{ if } xy = x.y, \\ 0 & \text{ else}. \end{cases} \end{equation}
If instead one considers the composition $\pa_{s_1} \pa_{s_2} \cdots \pa_{s_d}$ for a non-reduced expression, one gets zero.

Recall that $\leftdes(x)$ denote the left descent set of $x \in W$. The kernel of $\pa_s$ is precisely $R^s$. Using this and \eqref{composeDem} one can see that the following three statements are equivalent, for an element $x \in W$:
\begin{enumerate} \item $s \in \leftdes(x)$.
\item $\pa_s \pa_x = 0$.
\item $\pa_x(f) \in R^s$ for all $f \in R$.
\end{enumerate}
The following lemma is an immediate consequence.
\begin{lem} \label{lem:imageDem} The following are equivalent, for a finitary subset $I \subset S$.
\begin{enumerate} \item $I \subset \leftdes(x)$.
\item $\pa_s \pa_x = 0$ for all $s \in I$.
\item $\pa_x(f) \in R^I$ for all $f \in R$.
\end{enumerate}
\end{lem}

\begin{notation} Let $I \subset S$ be finitary, and recall that $w_I$ is the longest element of $W_I$. Then set $\pa_I := \pa_{w_I}$. \end{notation}

By Lemma \ref{lem:imageDem}, $\pa_I$ is a linear map from $R$ to $R^I$, and is $R^I$-linear since each $\pa_s$ is $R^s$-linear. An alternate description of $\pa_I$ was given by Demazure \cite[Proposition 3b]{Demazure}. For an English-language exposition, see \cite[Chapter 24, p. 463-464]{GBM}.

\begin{thm} \label{thm:Demazure} For any $I \subset S$ finitary we have
\begin{equation} \pa_I(f) = \frac{\sum_{w \in W_I} (-1)^{\ell(w)} w(f)}{\prod_{\alpha \in \Phi^+_I} \alpha}. \end{equation}
The denominator is the product of the positive roots for $W_I$, i.e. the set $$\Phi^+_I:=\{w(\alpha_s)\, \vert \, s\in I, w\in W_I \ \mathrm{and}\ \ell(ws)=
\ell(w)+1 \}.$$ \end{thm}

\begin{defn}\label{def.Demsurj} A realization satisfies \emph{Demazure surjectivity} if $\pa_s \colon R \to R^s$ is surjective for all $s \in S$. It satisfies \emph{generalized Demazure surjectivity} if $\pa_I \colon R \to R^I$ is surjective for all finitary $I \subset S$. \end{defn}

\begin{lem} \label{lem:productofrootsworks} Generalized Demazure surjectivity holds whenever the realization is faithful and $\Bbbk$ is a field of characteristic zero. \end{lem} 

\begin{proof} Let $\mu_I$ denote the product of the positive roots for $W_I$. 
When the realization is faithful, the positive roots are in bijection with the set of reflections in $W_I$.  Then the span of $\mu_I$ affords the sign representation. One can deduce from Theorem~\ref{thm:Demazure} that $\pa_I(\mu_I)$ equals the size of $W_I$, viewed as an element of $\Z \subset \Bbbk$. Since the image of $\pa_I$ is an $R^I$-submodule of $R^I$, and contains an invertible scalar, it must be everything.\end{proof}

The reason to care about generalized Demazure surjectivity is that it is equivalent to a more interesting property: that the ring extension $R^I \subset R$ is graded Frobenius, with trace map $\pa_I$, when $I$ is finitary. Moreover, when $I \subset J \subset S$ are both finitary, $R^J \subset R^I$ is a graded Frobenius extension, with trace map
\begin{equation} \label{paJIdefn} \pa^I_J := \pa_{w_J w_I^{-1}}. \end{equation}
For more details, see \S\ref{sec:frob}.

\subsection{NilCoxeter algebra as associated graded} \label{ss:nilCoxeter}

For $(W,S)$ a Coxeter system, let $(W,*,S)$ denote the \emph{Coxeter $*$-monoid}. It is generated by $s \in S$, modulo the braid relations and the \emph{$*$-quadratic relation}
\begin{equation} \label{starquadordinary} s * s = s. \end{equation}
Its elements agree with those of $W$, where $x \in W$ corresponds to $s_1 * s_2 * \cdots * s_d$ for a reduced expression. It is known that
\begin{equation} \label{starinequality} \ell(w * x) \le \ell(w) + \ell(x), \end{equation}
with equality if and only if $w * x = wx$ if and only if $w.x = wx$, see \cite[Lem. 3.2]{EKo}.
When this monoid is linearized into an algebra, it is 
equal to the Hecke algebra at $v=0$ \cite[Ch. 1]{Mathas}, known as the \emph{$0$-Hecke algebra}.

We wish to draw the reader's attention to the following observation. The $0$-Hecke algebra is filtered, where each element of $W$ lives in a degree equal to its length. This follows from \eqref{starinequality}. Thus one can consider the associated graded algebra. A priori, the associated graded algebra need not inherit the associated graded of the presentation of $(W,*,S)$, but in this case it does. If $\pa_s$ represents the image of $s$ in the associated graded, then \eqref{starquadordinary} is replaced by \eqref{nilquad}, while the braid relations are preserved. Thus the nilCoxeter algebra is the associated graded of the $0$-Hecke algebra.

In the rest of this chapter we consider a singular version of this phenomenon. In \cite[Definition 1.22]{EKo}, the Coxeter $*$-monoid is generalized to a category $\SC$, called the \emph{singular Coxeter monoid}, with one object for each finitary parabolic subset $I \subset S$. The morphisms from $J$ to $I$ are given by $(I,J)$-cosets, and composition acts on maximal elements by $*$-multiplication. In particular, $\End_{\SC}(\mt)$ is the Coxeter $*$-monoid. 

\begin{rem} One should not expect there to be a category similar to $\SC$ but where $\End(\mt)$ is the Coxeter group $W$ (rather than the $*$-monoid). This is because ``cosets are not invertible.'' \end{rem}

In the rest of this chapter we prove that the associated graded of $\SC$ matches a category constructed with Demazure operators, and we associate a Demazure operator to any double coset. As a consequence, the braid relations like \eqref{switchback} give rise to relations between Demazure operators.

\subsection{The nilCoxeter algebroid}

We will define a category  built out of Demazure operators and give three different  versions of it, that we will call $\Dem, \Dem'$ and $\Dem''$. We assume generalized Demazure surjectivity throughout.

\begin{defn} Let $p$ be an $(I,J)$-coset. Let $\pa_p := \pa_{\ma{p} w_J^{-1}}$, viewed as a linear map $R^J \to R^I$. \end{defn}

A priori, $\pa_{\ma{p} w_J^{-1}}$ is a linear map $R \to R$, but by the following lemma, it restricts to a map $R^J \to R^I$. The notation $\pa_p$ (as opposed to $\pa_{\ma{p} w_J^{-1}}$) indicates that we have restricted the domain of the function to $R^J$.

\begin{lem} \label{lem:sendsJtoI} The map $\pa_p$ sends $R^J$ to $R^I$. \end{lem}

\begin{proof} 
Since $\im\pa_J = R^J$ and $(\ma{p}w_J\inv).w_J=\ma{p}=w_I.(w_I\inv\ma{p})$, we have
\begin{equation} \im \pa_p= \im (\pa_p \circ \pa_{J}) = \im \pa_{\ma{p}}=\im (\pa_I \circ \pa_{w_I\inv \ma{p}})\subset \im \pa_I = R^I.\qedhere\end{equation}
\end{proof}

The following example is crucial.

\begin{ex}\label{ex:3.10} Suppose $I \subset J$ and that $p$ is the minimal $(I,J)$-coset, i.e., the double coset containing the identity element. Then $\ma{p} = w_J$, so $\pa_p = \pa_{\id}$ is the identity operator on $R$, and induces the inclusion map $R^J \hookrightarrow R^I$. If $q$ is the minimal $(J,I)$-coset, then $\ma{q} = w_J$ and $\pa_q = \pa_{w_J w_I^{-1}} = \pa^I_J$.
\end{ex}

\begin{defn} Let $(W,S)$ be a Coxeter system and $R$ be the polynomial ring of some realization. The \textit{nilCoxeter algebroid}  $\Dem = \Dem(W,S,R)$ is the following subcategory of $\Bbbk$-vector spaces. 
The objects are finitary subsets $I \subset S$, associated with the vector space $R^I$. 
The morphism space from $J$ to $I$ consists of $\Bbbk$-linear combinations of the operators $\pa_p$ associated to $(I,J)$-cosets $p$.  \end{defn}

Two important properties are not obvious from this definition.
\begin{itemize}
\item The set $\{\pa_p\}$ is linearly independent as $p$ ranges over all $(I,J)$-cosets. 
\item When $p$ is an $(I,J)$-coset and $q$ is a $(J,K)$-coset, then $\pa_p \circ \pa_q$ is in the span of Demazure operators for $(I,K)$-cosets. Thus composition is well-defined on $\Dem$.
\end{itemize}
To justify these properties, we will describe the same subcategory of vector spaces in a different way, and then prove that the two categories agree.

\begin{defn} \label{def:demprime}  Let $(W,S)$ be a Coxeter system and $R$ be the polynomial ring of some realization. Let $\Dem' = \Dem'(W,S,R)$ denote the following category, which is a subcategory of vector spaces. The objects are finitary subsets $I \subset S$, associated with the vector space $R^I$. The morphism space from $I$ to $J$ consists of linear combinations of compositions of the following generating morphisms:
\begin{enumerate}
    \item Whenever $Is$ is finitary, a morphism $Is \to I$ corresponding to the inclusion of rings $R^{Is} \hookrightarrow R^I$.
    \item Whenever $Is$ is finitary, a morphism $I \to Is$ corresponding to the Demazure operator $\pa^I_{Is} \co R^I \to R^{Is}$.
\end{enumerate}
That is, the morphisms in $\Dem'$ are all linear combinations of maps obtained as compositions of inclusion maps and Frobenius trace maps.
\end{defn}

\begin{notation} For a singular expression $I_{\bullet} = [I_0, \ldots, I_d]$, let $\pa_{I_{\bullet}}$ denote the corresponding composition of inclusion maps and Frobenius trace maps, a linear transformation from $R^{I_d}$ to $R^{I_0}$.\end{notation}

\begin{ex} Associated to the expression $[\mt, s, \mt, t, \mt, u, \mt]$ we have the endomorphism of $R$ given by $\pa_s \pa_t \pa_u$. \end{ex}

\begin{ex} Let $\{s,t,u\}$ be the simple reflections in type $A_3$, with $m_{su} = 2$. Associated to the expression $[st, s, su]$ we have the map $\pa_s \pa_t \co R^{su} \to R^{st}$.\newline
Associated to both $[s,\mt,t,\mt,s]$ and $[s,st,s]$ we have the map $\pa_s \pa_t \co R^s \to R^s$. \end{ex}

It is easier to describe the operator $\pa_{I_{\bullet}}$ using multistep expressions.

\begin{lem} Suppose that $I_{\bullet}$ corresponds to the multistep expression
\[ [[ I_0\subset K_1\supset I_1\subset K_2 \supset \cdots \subset K_m\supset I_m]]. \]
Then we have
\begin{equation} \label{multistepdemazure} \pa_{I_{\bullet}} = \pa_{w_{K_1} w_{I_1}^{-1}} \circ \pa_{w_{K_2} w_{I_2}^{-1}} \circ \cdots \circ \pa_{w_{K_m} w_{I_m}^{-1}}. \end{equation} \end{lem}

\begin{proof} We read the expression from right to left, and let $k$ range from $m$ down to $1$. The parabolic subgroup grows from $I_k$ to $K_k$, so one applies $\pa^{I_k}_{K_k} = \pa_{w_{K_k} w_{I_k}^{-1}}$. Then the parabolic subgroup shrinks from $K_k$ to $I_{k-1}$, which corresponds to an inclusion map (nothing happens). \end{proof}

\begin{prop} \label{prop:demred} If $I_{\bullet} \expr p$ is a reduced expression then $\pa_{I_{\bullet}} = \pa_p$. Otherwise, $\pa_{I_{\bullet}} = 0$. \end{prop}

\begin{proof} Recall from \eqref{map} that 
\[\ma{p} = w_{K_1} w_{I_1}^{-1} w_{K_2} \cdots w_{I_{m-1}}^{-1} w_{K_m}.\]
By definition (see \eqref{reduced}), if $I_{\bullet}$ is reduced then the lengths add in the expression
\[ (w_{K_1} w_{I_1}^{-1}) . (w_{K_2} w_{I_2}^{-1}) . \cdots . (w_{K_m} w_{I_m}^{-1}) = \ma{p} w_{I_m}^{-1}. \]
Thus by \eqref{multistepdemazure} and \eqref{composeDem}, $\pa_{I_{\bullet}} = \pa_p$ when $I_{\bullet}$ is reduced. Otherwise the lengths do not add, and the result is zero by \eqref{composeDem}. \end{proof}

\begin{cor}\label{cor:delpindep}
The set $\{\pa_p\}$, where $p$ ranges over all $(I,J)$-cosets, is a linearly independent collection of maps $R^J \to R^I$.
\end{cor}

\begin{proof}
Suppose there is a linear relation of the form $\sum a_p \pa_p = 0$, ranging over $(I,J)$-cosets $p$. By \cite[Prop. 4.13]{EKo}, we can precompose any $(I,J)$-coset $p$ with $[[J,\mt]]$ and postcompose with $[[\mt, I]]$, to obtain a reduced expression for the $(\mt, \mt)$-coset $\{\ma{p}\}$. Thus
\begin{equation} 0 = \pa_{[[\mt, I]]} \circ \left( \sum a_p \pa_p \right) \circ \pa_{[[J,\mt]]} = \sum a_p \pa_{\ma{p}}. \end{equation}
Since ordinary Demazure operators are linearly independent, we deduce that $a_p = 0$ for all $p$, as desired. 
\end{proof}

\begin{cor} Let $p$ be a $(I,J)$-coset and $q$ be a $(K,I)$-coset.
Then we have
\begin{equation} \label{composedemazureforcosets} \pa_q \circ \pa_p = \begin{cases} \pa_{q . p} & \text{ if } \ma{p} w_I^{-1} \ma{q} = (\ma{p} w_I^{-1}).\ma{q} = \ma{p} . (w_I^{-1} \ma{q}), \\ 0 & \text{ else}. \end{cases}\end{equation}
\end{cor}

\begin{proof} This follows immediately from Propositions \ref{prop:demred} and \ref{prop:concat}. \end{proof}

\begin{cor} \label{cor:DDequal} The nilCoxeter algebroid $\Dem$ is equal (as a subcategory of vector spaces) to $\Dem'$. Thus it is well-defined (i.e. closed under composition). The morphisms $\{\pa_p\}$ (ranging over all double cosets) form a basis for morphisms in $\Dem$. \end{cor} 

\begin{proof} The objects of  $\Dem$ and $\Dem'$ being the same and the composition of morphisms for both categories just being the composition of functions, to show their equality we only need to prove an equality of Hom spaces. 

Every double coset has a reduced expression, thus $\pa_p$ is a morphism in $\Dem'$ for all $p$. This implies that the morphisms in $\Dem$ are contained in those of $\Dem'.$  Meanwhile, by Example~\ref{ex:3.10}, $\Dem$ contains the generators in $\Dem'$, and by \eqref{composedemazureforcosets} every composition of the generators is either $\pa_p$, for some $p$, or zero. Thus the morphisms in $\Dem'$ are contained in those of $\Dem$.

The set $\{\pa_p\}$ spans all morphisms by definition of $\Dem$, and is linearly independent by Corollary \ref{cor:delpindep}. \end{proof}

\begin{thm} \label{thm:dempresent} Suppose the realization is faithful and balanced. The nilCoxeter algebroid has a presentation, with generators as in Definition \ref{def:demprime}. The relations are:
\begin{enumerate}
    \item The braid relations \eqref{braidrelns}.
    \item The \emph{(singular) nil-quadratic relation}
    \begin{equation} \pa_{[Is,I,Is]} = 0. \end{equation}
\end{enumerate}
A basis for $\Hom_{\Dem}(I,J)$ is $\{\pa_p\}$ ranging over all $(I,J)$-cosets $p$. Moreover, $\Dem$ is isomorphic to the associated graded of the singular Coxeter monoid with respect to the length filtration. \end{thm}

\begin{proof}
Let $\Dem''$ denote the category with this presentation. To show there is a full functor $\Dem'' \to \Dem'$ (sending generators to generators) we need to confirm that the relations hold in $\Dem' = \Dem$.

The nil-quadratic relation holds because  $\pa^I_{Is}$ is zero on the subring $R^{Is}$. This is because $\pa^I_{Is} = \pa_{w_{Is}} w_I^{-1}$ and $s \in \rightdes(w_{Is}w_I^{-1})$, so $\pa^I_{Is}$ begins with $\pa_s$. Meanwhile, $R^{Is}$ is contained in $R^s$ and is killed by $\pa_s$. 

The braid relations hold by Proposition~\ref{prop:demred}, since both sides of a braid relation are a reduced expression for the same double coset.

Let us define a morphism $\pa^{''}_p \in \Dem''$ by
\[ \pa^{''}_p := \pa^{''}_{I_{\bullet}} \]
for some arbitrarily chosen reduced expression $I_{\bullet} \expr p$ (where $\pa^{''}_{I_{\bullet}}$ is now the obvious composition of generators in $\Dem''$). By Theorem \ref{thm:matsumoto}, all such reduced expressions are related by the braid relations, so $\pa^{''}_p$ is well-defined in $\Dem''$. By Proposition \ref{prop:demred}, $\pa^{''}_p$ in $\Dem''$ is sent to $\pa_p$ in $\Dem$. By Corollary \ref{cor:delpindep}, $\{\pa^{''}_p\}$ is linearly independent in $\Dem''$ (since it is sent by a linear functor to a linearly independent set).

Meanwhile, if $I_{\bullet}$ is not reduced, then by Theorem \ref{thm:notreduced} we can apply a series of braid relations to get to an expression 
containing $[Js,J,Js]$ inside for some $J$ and $s$. Applying the relations in $\Dem''$, we see that $\pa_{I_{\bullet}} = 0$. Thus every composition of the generators in $\Dem''$ is either $\pa^{''}_p$ for some $p$, or zero. Therefore $\{\pa^{''}_p\}$ is a spanning set for $\Dem''$.

We conclude that $\{\pa^{''}_p\}$ is a basis for morphism spaces in $\Dem''$, and therefore the functor to $\Dem'$ is fully faithful.

Just as the monoid $(W,*,S)$ was filtered by length, so too the category $\SC$ is filtered by length, see \cite[\S 3.4]{EKo}. Composition in the associated graded of (the linearization of) $\SC$ of the images of $p$ and $q$ would give $p * q$ if the lengths add, and $0$ otherwise. Thus the composition rule agrees with \eqref{composedemazureforcosets}. The natural functor from $\Dem''$ to the associated graded is therefore well-defined and evidently fully faithful.
\end{proof}

\begin{rem} 
The category $\SC$ has a presentation with the braid relations and the $*$-quadratic relation $[Is,I,Is] = [I]$, see \cite[Theorem 5.32]{EKo}. Taking the $\Bbbk$-linearization of $\SC$ and then taking its associated graded, the presentation of $\SC$ becomes the presentation of $\Dem''$.
\end{rem} 

\subsection{Aside on cohomology and \texorpdfstring{$K$}{K}-theory}
 
When $(W,S)$ is a Weyl group of a complex semisimple Lie group $G$ with Borel subgroup $B$, the operators $\pa_I$ have a geometric construction. We set $R = H_B^*$ to be the $B$-equivariant cohomology of a point.
Let $P_I$ denote the parabolic subgroup of $G$ corresponding to $I \subset S$. Then $R^I = H_{P_I}^*$ is the $P_I$-equivariant cohomology of a point. The map from the $P_I$-classifying space to the $B$-classifying space is a $P_I/B$ bundle, and $P_I/B$ is a smooth projective variety. Therefore, there is a proper pushforward map $H_B^* \to H_{P_I}^*$, and using relative Poincar\'{e} duality, this is a Frobenius extension. Indeed, the proper pushforward map agrees with the Demazure operator $\pa_I$.

It is common instead to consider the equivariant $K$-theory of a point. Let $R_K := K_B^*$ and $R_K^I := K_{P_I}^*$. The proper pushforward in $K$-theory is a different operator, commonly called a ($K$-theoretic) Demazure operator or a Lusztig-Demazure operator. This operator satisfies the $*$-quadratic relation $\pa_s \circ \pa_s = \pa_s$, rather than the nil-quadratic relation.

It is straightforward to adapt the arguments of this chapter to construct a category whose morphisms are built from $K$-theoretic Demazure operators instead of ordinary Demazure operators. This category will be equivalent to the linearization of $\SC$, rather than the associated graded of the linearization.

\subsection{Asides on other kinds of realizations} \label{subsec:nastier}

In this technical series of remarks, we address what is known and likely to happen when the realization is not faithful, or when it is unbalanced.

When the realization is not faithful, the operators $\{\pa_x\}$ need not be linearly independent, and additional relations are needed to present the nilCoxeter algebra. The map $\pa_I$ is often zero (and thus generalized Demazure surjectivity will fail), when the realization is not faithful upon restriction to $W_I$. There is a full but not necessarily faithful functor $\Dem'' \to \Dem$, by the same proof. It is $\Dem''$ which should be the associated graded of the singular Coxeter monoid.

When the realization is not balanced and $m_{st}$ is odd, the braid relations \eqref{dembraid} will involve an invertible scalar. For example, if $m_{st} = 3$ and $a_{st} = q$ and  $a_{ts} = q^{-1}$, then $\pa_s \pa_t \pa_s = -q \pa_t \pa_s \pa_t$. See \cite[Claim A.7]{Bendihedral} for details.

As a consequence, different reduced expressions for $x \in W$ will produce Demazure operators which agree up to rescaling. To define $\pa_x$, we fix a choice of reduced expression. Relations like \eqref{composeDem} are understood to only hold up to (computable) scalars.

The operators $\{\pa_x\}_{x \in W}$ still span the nilCoxeter algebra. The literature does not contain a proof (for faithful, unbalanced realizations) that they are linearly independent, although a proof along the lines of Demazure's proof in \cite[Corollaire 1]{Demazure} should suffice. Demazure expresses $\pa_x$ as a linear combination of operators $w \in W$ over the fraction field of $R$, and uses this to deduce linear independence. Given the state of the literature, let us say that the linear independence of Demazure operators is conjectural for faithful unbalanced representations.

Generalized Demazure surjectivity should continue to hold even for faithful unbalanced realizations. The proof of Lemma \ref{lem:productofrootsworks} will not work, as positive roots are not well-defined. However, there is a distinguished collection of lines spanned by roots. One should replace the product of all positive roots with the product of one root from each root line, an element which is well-defined up to invertible scalar. With this modification the proof of Lemma \ref{lem:productofrootsworks} should work, though this is only conjectural.

It is unknown if there is a variant on the Coxeter $*$-monoid whose associated graded is the nilCoxeter algebra for an unbalanced realization.

Consider an unbalanced realization, with trace maps $\pa^I_J$ defined so that \eqref{compatible} holds. One can define $\Dem$ and $\Dem'$ in the same way as above. If $I_{\bullet} \expr p$, then $\pa_{I_{\bullet}}$ and $\pa_p$ will agree up to scalar. Consequently Corollary \ref{cor:DDequal} will still hold. The nil-quadratic relation will hold in $\Dem$, as will the up-up and down-down braid relations, but the switchback relation will only hold up to a scalar. Defining $\Dem''$ with the appropriate scalars built into the switchback relations, the proof of Theorem \ref{thm:dempresent} holds mutatis mutandis except the proof that $\Dem'' \to \Dem'$ is faithful, which relies on the conjecture above that $\{\pa_p\}$ is linearly independent.

Finally, for $I \subset J \subset K$ in the balanced case, one has
\begin{equation} \label{eq:compatibleearly} \pa^I_K = \pa^J_K \circ \pa^I_J. \end{equation}
In the unbalanced setting, this may only hold up to scalar, as currently stated. However, \eqref{eq:compatibleearly} is a desirable feature for chains of Frobenius extensions. Consequently, one should not use \eqref{paJIdefn}, but should redefine $\pa^I_J$ as some scalar multiple of $\pa_{w_J w_I^{-1}}$, in order that \eqref{eq:compatibleearly} holds. For an example of this sort of bookkeeping, see \cite[\S 3.2]{EQuantumI}.

\begin{rem} See \cite{EWLocalized} for further details on how the failure of faithfulness or balancedness affects the Hecke category. \end{rem}

%% file: Frobenius.tex
\section{New properties of some Frobenius extensions} \label{sec:frob}

\subsection{Demazure operators and Frobenius extensions} \label{subsec:frobextdemazure}

By a famous theorem of Chevalley and Shephard--Todd \cite{ShephardTodd}, if $I \subset S$ is finitary, then $R$ is a free graded module over $R^I$ of finite rank, and $R^I$ is itself a polynomial ring. An upgraded version of this theorem was proven by Demazure \cite{Demazure}, who gave an explicit description of the basis of $R$ over $R^I$ using Demazure operators. 

\begin{defn}\label{defn:frobext}  An inclusion $A \subset B$ of graded commutative rings is called a \emph{graded Frobenius extension of degree $\ell$} if it comes equipped with an \emph{nondegenerate} $A$-linear map $\pa^B_A \colon B \to A$ called the \emph{(Frobenius) trace}, which is homogeneous of degree $-2\ell$. Here, $\pa^B_A$ is called nondegenerate if there exist homogeneous finite bases $\{c_i\}$ and $\{d_i\}$ of $B$ as a free $A$-module such that \begin{equation} \label{tracenondeg} \pa^B_A(c_i d_j) = \delta_{ij}.\end{equation}
\end{defn}

Note that $\pa^B_A$ must be a surjective map $B \to A$: by $A$-linearity, the image is an $A$-submodule of $A$, and by \eqref{tracenondeg}, this image contains $1$.

The following theorem is a direct consequence of the results in \cite{Demazure}. An English-language exposition can be found in \cite[Th 24.36]{GBM}. We also re-explain the proof in the next section, because we will adapt it to prove new results.

\begin{thm}\label{thm.FrobI}  If $I \subset S$ is finitary and $\pa_I \colon R \to R^I$ is surjective, then $R^I \subset R$ is a graded Frobenius extension of degree $\ell(I)$, with trace map $\pa_I$.
\end{thm}

A more general version of this theorem is the following. Recall the definition of $\pa^I_J$ from \eqref{paJIdefn}.

\begin{thm}[{\cite[Th 24.40]{GBM} \label{thm:FrobExtIJ}}] If $I \subset J \subset S$ are both finitary and generalized Demazure surjectivity holds, then $R^J \subset R^I$ is graded Frobenius of degree $\ell(J) - \ell(I)$, with trace map $\pa^I_J$.
\end{thm}

 We will be recalling the techniques used to prove Theorem \ref{thm:FrobExtIJ}, because we need similar techniques to prove our new results. We begin by justifying that $\pa^J_I$ has the correct codomain, since a priori, $\pa^J_I$ is just an endomorphism of $R$. This is a consequence of Lemma~\ref{lem:sendsJtoI}, with the same proof, but we rewrite it in this simpler setting.

\begin{lem} \label{lem:paIJ} The map $\pa^I_J$ sends $R^I$ to $R^J$. Moreover, $\pa^I_J \colon R^I \to R^J$ is surjective. \end{lem}

\begin{proof} Because $\pa_I: R \to R^I$ is surjective, any element of $R^I$ has the form $\pa_I(f)$ for some $f \in R$. By \eqref{composeDem} we have $\pa^I_J(\pa_I(f)) = \pa_J(f)$, which lives in $R^J$ by Lemma~\ref{lem:imageDem}.

For any $g \in R^J$ there is some $f \in R$ such that $\pa_J(f) = g$. Then $g = \pa^I_J(\pa_I(f))$, so $g$ is in the image of $\pa^I_J$. \end{proof}

Note that
\begin{equation*} \label{compatible} \pa^I_K =\pa^J_K \circ  \pa^I_J \end{equation*}
whenever $I \subset J \subset K$.

\subsection{Almost dual bases}

We recall the proof of Theorem \ref{thm:FrobExtIJ}, following the exposition of \cite[Sec. 24.3]{GBM}.

Recall from Definition~\ref{defn:frobext} that a Frobenius extension (with trace map $\pa$) necessitates the existence of a pair of dual bases $\{c_i\}$ and $\{d_i\}$ for which $\pa(c_i d_j) = \delta_{ij}$. How would one construct dual bases for the extensions $R^I \subset R$, or the extensions $R^J \subset R^I$ when $I \subset J \subset S$? Why is the assumption of generalized Demazure surjectivity enough to imply that dual bases exist?

There is extensive literature on dual bases for the Frobenius extension $R^S \subset R$ when $R = \R[x_1, \ldots, x_n]$ and $W = S_n$. Outside of this case (and even for the geometric realization in type $A$, which is spanned by roots), there is no explicit construction of dual bases in the literature to our knowledge. It remains an important open problem to find a combinatorial construction of dual bases!

However, what is easier to find is an explicit construction for almost dual bases.

\begin{defn}\label{def:gradedfrobext}

Let $A\subset B$ be a graded extension of positively graded algebras over a field $\Bbbk$, for which $B$ is free of finite rank over $A$.  Assume that both are spanned by the identity element in degree zero. Let $\mg_A$ be the graded Jacobson radical of $A$, which is the ideal spanned by all elements of strictly positive degree.

Let $\pa_A^B:B\to A$ be an $A$-linear map of degree $-2\ell$. Finite homogeneous bases $\{c_i\}$ and $\{d_i\}$ (bases of $B$ as an $A$-module) are called \emph{almost dual bases} if
\begin{equation}\label{eq:almostdual} \pa_A^B(c_i d_j) \equiv \delta_{ij} + \mg_A. \end{equation}
\end{defn}

\begin{lem} \label{lem:eqgivesADB} Let $A$ and $B$ as in \Cref{def:gradedfrobext}.
Suppose that $\{c_i\}_{i=1}^n$ and $\{d_i\}_{i=1}^n$ are subsets of homogeneous elements of 
$B$ which satisfy \eqref{eq:almostdual} and suppose $B$ is free over $A$ of rank $n$. Then $\{c_i\}$ and $\{d_i\}$  are almost dual bases.
\end{lem}

\begin{proof}
It is enough to show that $\{c_i\}$ and $\{d_i\}$ are $A$-bases of $B$. We index the element of  $\{c_i\}_{i=1}^n$  in ascending order of their degrees, i.e.,  such that $\deg(c_i)\leq \deg(c_j)$ for $i<j$. We have $\deg(d_i)=2\ell-\deg(c_i)$ and, by degree reasons, we observe that \begin{equation}\label{realmeaning}\pa_A^B(c_id_j)=\begin{cases}0&\text{ if }i<j\\
1& \text{ if }i=j\\
f\in \mg_A&\text{ if }i>j
\end{cases}.\end{equation}

By assumption there exists a basis $\{x_i\}_{i=1}^n$ of $B$ over $A$. We can write $c_i=\sum \gamma_{ij}x_j$ and $d_i=\sum \delta_{ij}x_j$, where $\gamma:=(\gamma_{ij})$ and $\delta:=(\delta_{ij})$ are $n\times n$-matrices with coefficients in $A$. 
Then \eqref{realmeaning} says that the product $\gamma X \delta^{t}$ is a lower unitriangular matrix, where $X$ is the $n\times n$ matrix with $(i,j)$-entries $\pa^B_A(x_ix_j)$. In particular, both $\gamma$ and $\delta^t$ are invertible. It follows that $\{c_i\}$ and $\{d_i\}$ are bases.
\end{proof}

\begin{lem} \label{lem:ADBgivesDB} Suppose that $\{c_i\}$ and $\{d_i\}$ are almost dual bases of $B$ over $A$. Then $\{c_i\}$ and $\{d'_i\}$ are dual bases, for some $\{d'_i\}$. Moreover, $d_i \equiv d'_i$ modulo the ideal in $B$ generated by $\mg_A$. In particular, the extension $A\subset B$ is graded Frobenius.
\end{lem}

\begin{proof}
This follows by a straightforward application of the Gram--Schmidt algorithm. We can fix an indexing of the basis $\{c_i\}_{i=1}^n$ such that $\deg(c_i)\leq \deg(c_j)$ if $i<j$ and then set $d_n'=d_n$, $d_{n-1}'=d_{n-1} -\pa_A^B(c_nd_{n-1})d_n$, 
and so on. 
\end{proof}

\begin{rem} Suppose $S$ is finitary. Let $(R^S_+)$ denote the ideal in $R$ generated by all elements of $R^S$ in positive degree, and let $C = R / (R^S_+)$ be the so-called \emph{coinvariant ring}. Then $\Bbbk \subset C$ is a Frobenius extension, with trace map $C \to \Bbbk$ induced by $\pa_S$. Almost dual bases for $R^S \subset R$ descend to actual dual bases for $\Bbbk \subset C$. \end{rem}

Here is Demazure's construction of almost dual bases. First a quick lemma.

\begin{lem} Suppose that $x \in W$ and $s \in \rightdes(x)$. For any $f, g \in R$ we have \begin{equation} \label{eq:canshiftsover} \pa_x(\pa_s(f)g) = \pa_x(f \pa_s(g)).\end{equation}
\end{lem}

\begin{proof} The two key ingredients to this proof are the Leibniz rule
\begin{equation*} \pa_s(fg) = \pa_s(f) g + s(f) \pa_s(g) = \pa_s(f) g + s(f \pa_s(g)), \end{equation*}
and the anti-invariance of $\pa_s$
\begin{equation} \label{antiinvarianceofpas} \pa_s(s(f)) = -\pa_s(f). \end{equation}
Both are basic consequences of the definition of $\pa_s$.

We claim that $\pa_x(s(f)) = -\pa_x(f)$ as well, when $s \in \rightdes(x)$. This follows from \eqref{antiinvarianceofpas} because $x = w.s$ for some $w \in W$, and $\pa_x = \pa_w \pa_s$. We also claim that $\pa_x \pa_s = 0$, which follows for the same reason. Thus
\begin{equation*} 0 = \pa_x(\pa_s(fg)) = \pa_x(\pa_s(f) g + s(f \pa_s(g))) = \pa_x(\pa_s(f) g) - \pa_x(f \pa_s(g)).\qedhere \end{equation*}
\end{proof}

\begin{proof}[Proof of Theorem~\ref{thm.FrobI}]
By generalized Demazure surjectivity, we can assume the existence of an element $P_I \in R$ of degree $2\ell(I)$ such that
\begin{equation} \pa_I(P_I) = 1. \end{equation}

Let us compute $\pa_I(\pa_y(P_I) \pa_z(P_I))$ for various $y, z \in W_I$. If $\ell(y) + \ell(z) > \ell(I)$ then the result is zero for degree reasons. If $\ell(y) + \ell(z) < \ell(I)$ then the result is in $\mg_{R^I}$ for degree reasons. If $\ell(y) + \ell(z) = \ell(I)$ we perform the following analysis. Using \eqref{eq:canshiftsover} over a reduced expression for $y$, we obtain
\begin{equation*} \pa_I(\pa_y(P_I) \pa_z(P_I)) =
\pa_I(P_I \pa_{y^{-1}} \pa_z(P_I)) =
\pa_I(P_I) \pa_{y^{-1}} \pa_z(P_I) = \pa_{y^{-1}} \pa_z(P_I). \end{equation*}
In the second equality we pulled out $\pa_{y^{-1}} \pa_z(P_I)$ since it is a scalar (degree zero). If $y^{-1} . z = w_I$ then this scalar is $1$. Otherwise, the lengths do not add up and $\pa_{y^{-1}} \pa_z = 0$.

To conclude our calculation, we have
\begin{equation*} \pa_I(\pa_y(P_I) \pa_z(P_I)) \equiv \delta_{y^{-1} . z = w_I} + \mg_{R^I}. \end{equation*}
Here $\delta_{y^{-1} . z = w_I}$ is one if $y^{-1}.z = w_I$ (implying that the lengths add up) and zero otherwise.
Then by Proposition~\ref{prop:bases} below together with Lemma~\ref{lem:eqgivesADB}, the sets $\{\pa_y(P_I)\}_{y \in W_I}$ and $\{\pa_z(P_I)\}_{z \in W_I}$ are almost dual bases for $R$ over $R^I$, when indexed appropriately. To match bases we should set $z = y^\circ$ where $y^{-1} . y^\circ = w_I$.
\end{proof}

\begin{prop}\label{prop:bases}
    Let $I\subset J$ be finitary subsets of $S$. Then $R^I$ is free over $R^J$ with basis $\{\pa_I\pa_y(P_J)\}$ where $y$ ranges over minimal coset representatives in $W_I\backslash W_J$.
\end{prop}
\begin{proof}
In the case when $I=\emptyset$, this is proved in \cite[Prop. 5]{Demazure} if generalized Demazure surjectivity holds (the torsion index defined in \cite[\S 5]{Demazure} is indeed $1$ in this case).  Thus we can write any $f \in R$ uniquely as $f = \sum g_x \pa_x(P_J)$, where $g_x \in R^J$ and the sum is indexed over $x \in W_J$.

Let now $I$ be arbitrary, and $f \in R^I$. As above we can write $f =\sum g_x \pa_x(P_J)$. We have $\pa_s(f)=0$ for any $s \in I$, and $\pa_s(f) = \sum g_x \pa_s \pa_x(P_J)$. Either $\pa_s \pa_x = 0$ or $\pa_s \pa_x = \pa_{sx}$; we conclude that
\[ 0 = \pa_s(f) = \sum_{sx > x} g_x \pa_{sx}(P_J).\]
By linear independence we conclude that $g_x = 0$ whenever $sx>x$. Repeating this argument for all $s \in I$, we deduce $g_x = 0$ unless $I \subset \leftdes(x)$, so that $x = w_I . y$ for some minimal coset representative $y \in W_J$.

In conclusion, we can write 
\[ f =\sum g_{w_I.y} \pa_I \pa_y (P_J)\] and the claim follows.
\end{proof}

\begin{proof}[Proof of Theorem \ref{thm:FrobExtIJ}]
$I \subset J \subset S$ are finitary subsets.  
We claim that 
\[ \{\pa_I \pa_y(P_J)\}_{y \in W_I \backslash W_J} \qquad \text{and}  \qquad \{ \pa_I \pa_z(P_J)\}_{z \in W_I \backslash W_J}\]
satisfy \eqref{eq:almostdual}. Here we are using $\pa_y$ and $\pa_z$ to mean  $\pa_{\mi{y}}$ and $\pa_{\mi{z}}$, where $\mi{y}$ and $\mi{z}$ are the minimal coset representatives for the cosets $y$ and $z$. We do so to avoid cluttering the page with underlines. By Lemma~\ref{lem:eqgivesADB} and Proposition~\ref{prop:bases} the condition \eqref{eq:almostdual} are enough for the two sets to be almost dual bases for $R^I$ over $R^J$, when indexed appropriately.  

First, observe that these bases consist of homogeneous elements of $R^I$, being in the image of $\pa_I$. Second, we have
\begin{equation*} \pa^I_J{\color{red}\big(}\pa_I \pa_y(P_J) \pa_I \pa_z(P_J){\color{red}\big)} = \pa^I_J{\color{red}\big(}\pa_I {\color{green}\big(}\pa_y(P_J) \pa_I \pa_z(P_J) {\color{green}\big)}  {\color{red}\big)} = \pa_J{\color{green}\big(}\pa_y(P_J) \pa_I \pa_z(P_J) {\color{green}\big)}. \end{equation*}
The first equality holds because $\pa_I \pa_z(P_J)$ is in $R^I$ and $\pa_I$ is $R^I$-linear. The second equality holds because $\pa^I_J \circ \pa_I = \pa_J$. Now apply \eqref{eq:canshiftsover} with $x = w_J$ to deduce that
\begin{equation*} \pa_J(\pa_y(P_J) \pa_I \pa_z(P_J)) = \pa_J(P_J \pa_{y^{-1}} \pa_I \pa_z(P_J)). \end{equation*}
Using arguments similar to the previous proof, we deduce that 
\begin{equation} \label{dualpairingIJ} \pa^I_J(\pa_I \pa_y(P_J) \pa_I \pa_z(P_J)) \equiv \delta_{y^{-1} . w_I . z = w_J} + \mg_{R^J}. \end{equation}
Thus to index dual bases to match, we set $z = y^\circ$ where $y^{-1} . w_I . y^\circ = w_J$. Indeed, $y \mapsto y^\circ$ is a bijection on minimal left coset representatives for $W_I \backslash W_J$.
\end{proof}

\begin{rem} To our knowledge, this proof first appeared in \cite[Sec. 24.3]{GBM}. We have filled in some of the missing details above.
\end{rem}

\subsection{New results on almost dual bases}

Suppose that $I, J \subset S$ are such that $I \cup J$ is finitary. By Theorem~\ref{thm:FrobExtIJ}, $R^I \subset R^{I \cap J}$ is a Frobenius extension, so there exist dual bases relative to $\pa^{I \cap J}_I$. However, a stronger statement is true, to whit:
\begin{equation} \label{eq:star}\tag{$\star$} \begin{array}{c}\text{There are dual bases for $R^{I \cap J}$ over $R^I$  for which}\\\text{one of the bases lives in the subring $R^{J}$.}\end{array} \end{equation}
This condition first arose in \cite{ESWFrob}, where it was also denoted $(\star)$. It is a prerequisite for the diagrammatic technology of \cite{ESWFrob} to apply to a square of Frobenius extensions.  The condition \eqref{eq:star} may fail without the assumption that $I \cup J$ is finitary. See Example~\ref{affin}. 

In Theorem~\ref{thm:dualbasisinimage} we prove a stronger result, of which \eqref{eq:star} is a special case (see Example~\ref{ex.star}). In a sequel to this paper we construct a basis for morphisms between singular Soergel bimodules. Theorem~\ref{thm:dualbasisinimage} is crucial in our proof that the purported basis is linearly independent. Some examples and additional discussion can be found in the next section.

First, a quick lemma.

\begin{lem} \label{lem:mipsendsJtoK} Let $p$ be an $(I,J)$-coset and let $K = \leftred(p)$. Then $\pa_{\mi{p}}$ sends $R^J$ to $R^K$. \end{lem}

\begin{proof} We give two quick proofs. For the first proof, note that $R^J$ is the image of $\pa_J$, and that $\mi{p}. w_J$ has $K$ in its left descent set. Then $\pa_{\mi{p}} \pa_J = \pa_K \pa_w$ for some $w \in W$, and the image lives in $R^K$.

For the second proof, recall from Theorem~\ref{thmA} that $p = [[I \supset K]] . p'$, where $p'$ is a particular $(K,J)$-coset. By Equation \eqref{upper}, we see that  $\ma{p} = (w_I w_K^{-1}).\ma{p'}$, and with Equation \eqref{mapmip} we obtain  $\ma{p'} = \mi{p} . w_J$. Thus $\pa_{p'} = \pa_{\mi{p}}$, and it is a map from $R^J$ to $R^K$ by Lemma~\ref{lem:sendsJtoI}. \end{proof}

\begin{lem} \label{lem:almostdualexpl}  Let $I, J, L \subset S$ be such that $I \cup J \subset L$ and $L$ is finitary. Pick $P_I, P_L \in R$ such that $\pa_I(P_I) = 1$ and $\pa_L(P_L) = 1$. Let $p$ be an $(I,J)$-coset contained in $W_L$, and let $K = \leftred(p)$. Let $y \in W_I$ be a minimal representative for its coset $W_{K} y$, and let $z \in W_L$ be arbitrary. We say that $z$ is \emph{dual} to $y$ (relative to $p$) if
\begin{equation} \label{eq:circdual} y^{-1} . \mi{p} . w_J . z = w_L. \end{equation}
Then $z$ is dual to $y$ if and only if $z = y^{\circ}$, where
\begin{equation}\label{eq:ycirc} y^{\circ} = w_J \mi{p}^{-1} y w_L. \qedhere\end{equation}
Moreover,
\begin{equation} \label{eq:leadstoalmostdual} \pa^{K}_I(\pa_{K} \pa_y(P_I) \cdot \pa_{\mi{p}} \pa_J \pa_z(P_L)) = \delta_{z,y^{\circ}} + \mg_{I}. \end{equation}
Here, $\mg_I$ represents the ideal of positive degree elements in $R^I$.
\end{lem}

\begin{proof}
We use an argument similar to those in the previous section.
We consider the pairing
\[ \pa^K_I(\pa_K \pa_y(P_I) \pa_{\mi{p}} \pa_J \pa_z(P_L)).\]
Here, $y$ will range over minimal right coset representatives in $W_K \backslash W_I$; 
we have already seen in the previous section that $\{\pa_K \pa_y(P_I)\}$ is one of a pair of almost dual bases for $R^K$ over $R^I$.

By Lemma~\ref{lem:mipsendsJtoK}, $\pa_{\mi{p}} \pa_J$ has image contained in $R^K$. Since $\pa_K$ is $R^K$-linear, we have
\begin{align*}  \pa^K_I{\color{red}\big(}\pa_K \pa_y(P_I) &\pa_{\mi{p}} \pa_J \pa_z(P_L){\color{red}\big)} =
\pa^K_I{\color{red}\big(}\pa_K {\color{green}\big(}\pa_y(P_I) \pa_{\mi{p}} \pa_J \pa_z(P_L) {\color{green}\big)} {\color{red}\big)} \\ & =
\pa_I{\color{green}\big(}\pa_y(P_I) \pa_{\mi{p}} \pa_J \pa_z(P_L) {\color{green}\big)}. \end{align*}
Since $y \in W_I$, we can use \eqref{eq:canshiftsover} to prove
\begin{equation*} \pa_I(\pa_y(P_I) \pa_{\mi{p}} \pa_J \pa_z(P_L)) =
\pa_I(P_I \pa_{y^{-1}} \pa_{\mi{p}} \pa_J \pa_z(P_L)). \end{equation*}

If $y^{-1}. \mi{p} . w_J . z$ is not a reduced composition then the result is zero. So assume that $y^{-1} . \mi{p} . w_J . z = y^{-1} \mi{p} w_J z$. If $\ell(y^{-1} . \mi{p} . w_J . z) > \ell(L)$ then the result is zero for degree reasons, and if $\ell(y^{-1} . \mi{p} . w_J . z) < \ell(L)$ then the result lives in $\mg_{R^I}$ for degree reasons. If $\ell(y^{-1} . \mi{p} . w_J . z) = \ell(L)$ then we must have $y^{-1} . \mi{p} . w_J . z = w_L$, in which case
\begin{equation*} \pa_I(P_I \pa_{y^{-1}} \pa_{\mi{p}} \pa_J \pa_z(P_L)) = \pa_I(P_I) = 1. \end{equation*}
Thus we deduce that
\begin{equation} \pa^K_I(\pa_K \pa_y(P_I) \pa_{\mi{p}} \pa_J \pa_z(P_L)) = \delta_{y^{-1} . \mi{p} . w_J . z = w_L} + \mg_{R^I}. \end{equation}

Since $y$ runs over the minimal right coset representatives for $W_K\backslash W_I$, the element $y^{-1}$ runs over the minimal left coset representatives for $ W_I/W_K$. Thus by \cite[Lemma 2.12]{EKo} the composition $x = y^{-1}.\mi{p}.w_J$ is always reduced, and ranges amongst all elements $x$ of $p$ for which $\mi{p} w_J \le x \le \ma{p}$. Since $x \le \ma{p} \le w_L$, there exists a unique $z$ such that $x.z = w_L$. Moreover, we have $z=y^\circ$ for $y^\circ=w_J\mi{p}^{-1}yw_L$.
\end{proof}


\begin{thm}\label{thm:almostdualinimage}
Let $I, J, L \subset S$ be such that $I \cup J \subset L$ and $L$ is finitary. Let $p$ be an $(I,J)$-coset contained in $W_L$, and let $K = \leftred(p)$. Then 
\begin{equation}\label{eq:almostdualbasis}
    \{\pa_K \pa_y(P_I)\},\quad \{\pa_{\mi{p}} \pa_J \pa_{y^\circ}(P_L)\}
\end{equation} are almost dual bases of $R^K$ over $R^I$, where $y$ ranges among minimal coset representatives for right cosets in $W_K \backslash W_I$ and $y^{\circ}$ is as in \eqref{eq:ycirc}.
\end{thm}
\begin{proof}
It follows by \Cref{lem:almostdualexpl} that the two sets satisfy \eqref{eq:almostdual}. We also notice that $|W_K\backslash W_I|=\rank_{R^I}R^K$ by \Cref{prop:bases}, so they are almost dual bases by \Cref{lem:eqgivesADB}.
\end{proof}

\begin{thm} \label{thm:dualbasisinimage}
Let $I, J, L \subset S$ be such that $I \cup J \subset L$ and $L$ is finitary. Let $p$ be an $(I,J)$-coset contained in $W_L$, and let $K = \leftred(p)$. Then there exists a pair of dual bases for $R^K$ over $R^I$, where one of the bases is in the image of $\pa_{\mi{p}}$ when viewed as a map $R^J \to R^K$.
\end{thm}
\begin{proof}
    By Lemma~\ref{lem:ADBgivesDB}, the claim follows from Theorem~\ref{thm:almostdualinimage} since the second basis in \eqref{eq:almostdualbasis} is in the image of $\pa_{\mi{p}}$.
\end{proof}

\subsection{Additional remarks and examples}

First, some examples of Theorem~\ref{thm:dualbasisinimage}.

\begin{ex}
Let $W=S_4$ and $S=\{s_1,s_2,s_3\}$. Let $I=\{s_1,s_3\}$ and consider the $(I,I)$-coset $p$ with minimal element $s_2$, so that $\partial_{\mi{p}}=\partial_2$. Let $P_S\in R$ with $\partial_S(P_S)=1$ and let $P_I=\partial_{w_I^{-1}w_S}(P_S)=\partial_{2132}(P_S)$. Notice that $\partial_I(P_I)=\partial_I\partial_{w_I^{-1}w_S}(P_S)=1$.

We have $\leftred(p)=\emptyset$. We can apply $\partial_{\mi{p}}$ to a strict subset of the basis $\{\partial_{13}\partial_z(P_S)\}$ of $R^I$ over $R^S$ (where $z$ ranges over $W_I \backslash W$) to obtain a  basis of $R$ over $R^I$. In fact, we have
\begin{align*}
\partial_2\partial_{13}\partial_2(P_S) &=P_I\\
\partial_2\partial_{13}\partial_{21}(P_S) =\partial_3\partial_{2132}(P_S)&=\partial_3(P_I)\\
\partial_2\partial_{13}\partial_{23}(P_S) =\partial_1\partial_{2132}(P_S)&=\partial_1(P_I)\\
\partial_2\partial_{13}\partial_{231}(P_S) =\partial_{13}\partial_{2132}(P_S)&=\partial_{13}(P_I)=1.
\end{align*}
\end{ex}

\begin{ex}\label{ex.star} Let $I, J \subset S$ be such that $I \cup J$ is finitary, and let $p$ be the $(I,J)$-coset containing the identity. Then $\leftred(p) = I \cap J$ and $\pa_{\mi{p}}$ is the identity map. Thus Theorem~\ref{thm:dualbasisinimage} is the same statement as \eqref{eq:star} in this case.
\end{ex}

One can think of Theorem~\ref{thm:dualbasisinimage} as being the correct generalization of \eqref{eq:star} beyond identity double cosets.

For sake of discussion, let us isolate the conclusion of Theorem~\ref{thm:dualbasisinimage} from its hypotheses.
\begin{equation}\label{eq:starstar}\tag{$\star \star$}\begin{array}{c}
\text{For all $(I,J)$-cosets $p$, letting $K=\leftred(p)$, there exists a}\\
\text{pair of dual bases for $R^K$ over $R^I$ where one of the bases is }\\
\text{in the image of $\pa_{\mi{p}}$ when viewed as a map $R^J \to R^K$.}
\end{array}\end{equation}

Our final task is to discuss potential generalizations of Theorem~\ref{thm:dualbasisinimage} to infinite Coxeter groups. Recall that \eqref{eq:star} came with the hypothesis that the ambient parabolic subgroup $I \cup J$ is finitary, and \eqref{eq:starstar} was proven in Theorem~\ref{thm:dualbasisinimage}  under the hypothesis that the ambient parabolic subgroup $L$ is finitary. Both \eqref{eq:star} and \eqref{eq:starstar} make sense without these finiteness hypotheses, but may or may not be true depending on the realization. We are interested in additional hypotheses on the realization which might imply these conclusions.

For illustration let us discuss \eqref{eq:star} in the example of a dihedral group with simple reflections $s$ and $t$. We focus on the statement
\begin{equation}\label{eq:diehdralstar}\tag{$\star_{s,t}$} \qquad \begin{array}{c}\text{There are dual bases for $R$ over $R^s$  for which}\\\text{one of the bases lives in the subring $R^t$.}\end{array} \end{equation}

\begin{lem} In the context above of the dihedral group, \eqref{eq:diehdralstar} is equivalent to the existence of a polynomial $P$ with $\pa_s(P) = 1$ and $\pa_t(P) = 0$, i.e. the existence of a \emph{fundamental weight}. \end{lem}

\begin{proof} If $P$ exists then $\{1,P\}$ is a basis for $R$ over $R^s$ living in $R^t$. The dual basis is $\{\al_s - P, 1\}$. Note that $\{\al_s - P, 1\}$ does not live in $R^t$ for $m_{st} > 2$. Only one of the two bases is expected to live in $R^t$.

The converse is obvious. \end{proof}

\begin{rem} Note that Demazure surjectivity implies the existence of a polynomial $P$ for which $\pa_s(P) = 1$, but not one which simultaneously satisfies $\pa_t(P) = 0$.\end{rem}

Whether or not \eqref{eq:diehdralstar} holds for the infinite dihedral group depends on the choice of realization.

\begin{ex} Any realization of a finite dihedral group is also a realization of the infinite dihedral group. Here \eqref{eq:diehdralstar} holds by Theorem~\ref{thm:dualbasisinimage}. \end{ex}

\begin{ex}\label{affin} Consider the realization spanned by simple roots, whose Cartan matrix is the usual affine Cartan matrix of type $\tilde{A}_1$. Then $\alpha_s + \alpha_t \in R^{st}$, and spans the linear terms in both $R^s$ and $R^t$. Consequently, there is no fundamental weight, and \eqref{eq:diehdralstar} fails. \end{ex}

Let us isolate a condition on the realization which is equivalent to \eqref{eq:diehdralstar} and separates these two examples.

\begin{lem} \label{lem:starifHsnotHt} Continuing the notation above, let $H_s$ and $H_t$ denote the hyperplanes in the realization $V$ which are fixed by $s$ and $t$ respectively. Then \eqref{eq:diehdralstar} holds if and only if $H_s \ne H_t$. \end{lem}

\begin{proof} If $H_s = H_t$ then no fundamental weight can exist, as e.g. in Example~\ref{affin}.

If $H_s \ne H_t$ then $V = H_s + H_t$, so $R$ is generated by the subrings $R^s$ and $R^t$. Then the linear polynomials in $R$ are spanned by those in $R^s$ and $R^t$. By Demazure surjectivity there is some linear polynomial $f + g$ with $\pa_s(f+g) = 1$, where we assume $f \in R^s$ and $g \in R^t$. Then $\pa_s(g) = 1$ as well, so $g$ is a fundamental weight. \end{proof}

Thus the distinctness of the hyperplanes $H_s$ and $H_t$ is equivalent to \eqref{eq:diehdralstar}, which is a special case of \eqref{eq:starstar}. This motivates why one might expect a relationship between \eqref{eq:starstar} and reflection faithfulness.

\begin{defn} For a reflection $r$ (i.e. a $W$-conjugate of a simple reflection), let $H_r \subset V$ denote the subspace of $V$ fixed by $r$.  A realization is called \emph{reflection faithful} if $H_r$ is a hyperplane for all reflections $r$, and $H_r \ne H_{r'}$ whenever $r \ne r'$. \end{defn}

This condition was the original condition imposed by Soergel for Soergel bimodules to be well-behaved \cite{Soe07}, and is common in the literature (e.g., \cite{EWhodge}). For example, Libedinsky \cite{LL} uses reflection faithfulness to guarantee that double leaves form a basis for morphisms between Bott-Samelson bimodules.  Just as reflection faithfulness is the condition commonly used to guarantee the proper behavior of Soergel bimodules, \eqref{eq:starstar} is what we use in future work to guarantee proper behavior of singular Soergel bimodules.

\begin{conj} \label{conj:dualbasisinimage} Suppose the realization is reflection faithful. Then \eqref{eq:starstar} holds for all double cosets $p$. \end{conj}

\begin{rem} In the dihedral setting, one thinks of Conjecture~\ref{conj:dualbasisinimage} as being a melange of Lemma~\ref{lem:starifHsnotHt} with the additional ingredient of conjugation by $\mi{p}$. \end{rem}

\begin{rem} Perhaps reflection faithfulness is equivalent to \eqref{eq:starstar}, though it seems likely that a weaker condition would suffice. It would be interesting to find a condition on the distinctness of certain hyperplanes which is equivalent to \eqref{eq:starstar} in general. \end{rem}